\documentclass[12pt]{amsart}
\usepackage{color}
\usepackage[all,cmtip]{xy}
\usepackage{amssymb}
\usepackage{amsmath}
\usepackage{amsthm}
\usepackage{amscd}
\usepackage{amsfonts}
\usepackage{graphicx}
\usepackage{multicol}

\calclayout

\newtheorem{dummy}{anything}[section]
\newtheorem{theorem}[dummy]{Theorem}

\newtheorem{lemma}[dummy]{Lemma}
\newtheorem{proposition}[dummy]{Proposition}

\theoremstyle{definition}
\newtheorem{definition}[dummy]{Definition}

    \newtheorem{problem}[dummy]{Problem}
  \newtheorem*{acknowledgement}{Acknowledgement}

\newcommand{\cF}{\mathcal F}

\newcommand{\cH}{\mathcal H}

\newcommand{\cK}{\mathcal K}

\newcommand{\bbC}{\mathbb C}

\newcommand{\bbS}{\mathbb S}
\newcommand{\bbZ}{\mathbb Z}

\DeclareMathOperator{\Hom}{Hom} 
\DeclareMathOperator{\Aut}{Aut}

\DeclareMathOperator{\Ind}{Ind} 

\DeclareMathOperator{\Res}{Res} 
\DeclareMathOperator{\Inf}{Inf}
 
\DeclareMathOperator{\Isot}{Isot}
\DeclareMathOperator{\Isog}{Iso}

\DeclareMathOperator{\ind}{Ind}

 \DeclareMathOperator{\rk}{rk}

\newcommand{\iso}{\cong}
\newcommand{\G}{\Gamma}

\newcommand{\Fa}{\cH}  
\newcommand{\Fu}{\cF}   

 \DeclareMathOperator{\Rep}{Rep}

\newcommand{\leftexp}[2]{{\vphantom{#2}}^{ #1}{\hskip-1pt#2}}

\def\maprt#1{\smash{\,\mathop{\longrightarrow}\limits^{#1}\,}}

\begin{document}

\title{Fusion systems and group actions with abelian isotropy subgroups}

\author{\" Ozg\" un \" Unl\" u and Erg{\" u}n Yal\c c\i n }

\address{Department of Mathematics, Bilkent
University, Ankara, 06800, Turkey.}

\email{unluo@fen.bilkent.edu.tr \\
yalcine@fen.bilkent.edu.tr}


\thanks{Both of the authors are partially supported by T\" UB\. ITAK-TBAG/110T712.}


\begin{abstract}
We prove that if a finite group $G$ acts smoothly on a manifold $M$ so that all the isotropy subgroups are abelian groups with rank $\leq k$, then $G$ acts freely and smoothly on $M \times \bbS^{n_1} \times \dots \times \bbS^{n_k}$ for some positive integers $n_1,\dots, n_k$. We construct these actions using a recursive method, introduced in an earlier paper, that involves abstract fusion systems on finite groups. As another application of this method, we prove that every finite solvable group acts freely and smoothly on some product of spheres with trivial action on homology.  
\end{abstract}

\maketitle

\section{Introduction and statement of results}
\label{section:Introduction}Madsen-Thomas-Wall \cite{madsen-thomas-wall} proved that a finite group $G$ acts freely on a sphere if  (i) $G$ has no subgroup
isomorphic to the elementary abelian group $\bbZ /p \times \bbZ /p$ for any prime number $p$ and (ii) has no subgroup isomorphic to the dihedral group $D_{2p}$ of order $2p$ for any odd prime number $p$. At that time the necessity of these conditions were already known: P. A. Smith \cite{smith} proved that the condition (i) is necessary and Milnor \cite{milnor} showed the necessity of the condition (ii).

As a generalization of the above problem
we are interested in the
problem of characterizing those finite groups which can act freely and smoothly on a product of $k$ spheres for a given positive integer $k$. In the case $k=1$,
Madsen-Thomas-Wall uses surgery theory and considers the unit spheres of linear representations of subgroups to show that certain surgery obstructions vanish. In the case $k\geq 2$, one might hope
for the same method to work. However one would need to construct
free actions for certain families of groups where
the surgery obstructions could be detected
when they are reduced to subgroups in these families.
Some general methods which can be used to do constructions for some of these families are given in \cite{adem-davis-unlu}, \cite{unlu-yalcin1},  and \cite{unlu-yalcin2}. Our goal in this paper is to improve the methods used in \cite{unlu-yalcin2} and as a result obtain some new actions.

In the case $k=2$, Heller \cite{heller} showed that if
a finite group $G$ acts freely on a product of two spheres, then $G$ must have $\rk (G)\leq 2$ where $\rk (G)$, rank of $G$, denotes the maximum rank of elementary abelian subgroups $(\bbZ/p)^r$ in $G$. In \cite{unlu-yalcin2} we showed that this result has a converse for finite
$p$-groups, more precisely, we proved that a finite $p$-group $G$ acts freely and smoothly on a product of two spheres if and only if $\rk (G)\leq 2$. In \cite{unlu-yalcin2} we also proved that if a finite group $G$ acts smoothly on a manifold $M$ so that all the isotropy subgroups of $M$ are elementary abelian groups with rank $\leq k$, then $G$ acts freely and smoothly on $M \times
\bbS^{n_1} \times \dots \times \bbS^{n_k}$ for some positive integers $n_1,\dots, n_k$. This, in particular, implies that every extra-special $p$-group of rank $k$ acts freely and smoothly on a product of $k$ spheres.

To prove the results mentioned above, in \cite{unlu-yalcin2} we introduced a recursive method for constructing
group actions on products of spheres. The main idea of this recursive method is
to start with an action of a group $G$ on a manifold $M$ and obtain a new
action of $G$ on $M\times \mathbb{S}^N$ for some positive integer $N$ by gluing unit spheres of $\bbC G_x$-modules $V_x$ over $x\in M$ where $G_x$ denotes the isotropy subgroup of $G$ at $x$. In the construction we use a theorem of L\" uck-Oliver \cite{luck-oliver} on equivariant vector bundles. To apply the L\" uck-Oliver theorem, one needs to find a finite group $\Gamma$  which has a complex representation $V$ such that  the representations $V_x$ of isotropy subgroups $G_x$ can be obtained from $V$ by pulling back via a compatible family of homomorphisms $\alpha_x: G_x \to \Gamma$ (see Section \ref{section:ConstructionsOfFreeActions} for more details).

In  \cite{unlu-yalcin2} we introduced a systematic way of constructing a finite group $\G $ satisfying the properties above using abstract fusion systems.  The idea is the following: we first map all the isotropy groups into a finite group $S$, then we study the fusion system $\cF$  on $S$ that comes from the conjugations in $G$ and different choices of embeddings into $S$ (see Definition \ref{def:fusionsystems}  for a definition of a fusion system). Then, we use a theorem of S. Park \cite{park} to find $\Gamma$ as the automorphism group of a left $\cF$-stable  $S$-$S$-biset (see Definition \ref{lemma:FstableBisets}).

In this paper, we improve the method described above by proving a result that allows us to construct left $\cF$-stable bisets using strongly $\cF$-closed subgroups of $S$ (see Definition \ref{def:stronglyclosed} and Proposition
\ref{prop:Fchar_leftFbiset}). This result is particularly useful when $S$ is an abelian group. As a consequence, we obtain the following:

\begin{theorem}\label{theorem:AbelianIsotropy}
Let $G$ be a finite group acting smoothly on a manifold $M$. If all the isotropy subgroups of $M$ are abelian groups with rank $\leq k$, then $G$ acts freely and smoothly on $M \times \bbS^{n_1} \times \dots \times \bbS^{n_k}$ for some positive integers $n_1,\dots, n_k$.
\end{theorem}

A similar result has been proved recently in the homotopy category by Michele  Klaus \cite{klaus} using $G$-fibrations. This new tool for constructing $\cF$-stable $S$-$S$-bisets also makes it possible to tackle the following problem:

\begin{problem}\label{prob:MainProblem}
Does every finite group act freely on some product of spheres with trivial action on homology?
\end{problem}

It is known that every finite group can act freely on some product of spheres. This follows from a general construction given in \cite[page 547]{oliver} (which is attributed to J. Tornehave by Oliver). However, in this construction the free action is obtained by permuting the spheres, so the induced action on homology is not trivial. To find a free homologically trivial action of a finite group $G$ on a product of spheres, one may try to take the product of unit spheres $\bbS(V)$ over some family of complex representations $V$ of $G$. Actions that are obtained in this way are called linear actions on products of spheres. 
By construction linear actions are homologically trivial but not every finite group has a free linear action on a product of spheres. In fact, Urmie Ray \cite{ray} shows that finite groups that have such actions are rather special: If a finite group $G$ has a free linear action on a product of spheres, then all the nonabelian simple sections of $G$ are isomorphic to $A_5$ or $A_6$. 

Note that even solvable groups may not have free linear actions on a product of spheres. For example, the alternating group $A_4$ does not act freely on any product of equal dimensional spheres with trivial action on homology (see \cite[Theorem 1]{oliver}). So, it can not act freely on a product of linear spheres. On the other hand, it is easy to construct homologically trivial actions on products of spheres for supersolvable groups and many other groups with small rank and fixity (see \cite{ray} and \cite{unlu-yalcin1}).
 
To construct a free homologically trivial action on a product of spheres, one can start with an action on some product of spheres, say on a product of all linear spheres, then apply the fusion system methods described above to obtain an action on $M \times \bbS ^N$ for some $N$ so that the total size of isotropy subgroups is smaller. Continuing this way, one can inductively construct a free homologically trivial action on a product of spheres as long as at each step a compatible family of representations coming from a representation of a finite group $\Gamma$ can be found. So far, we could do this only for solvable groups:

\begin{theorem}\label{theorem:FreeActionsOfSolvableGroups}
Any finite solvable group can act freely on some product of spheres with trivial action on homology.
\end{theorem}  

The proof follows from special properties of maximal fusion systems on abelian groups. To extend the proof to all finite groups, one would need to study the conditions on a fusion system $\cF$ that would imply the existence of a left $\cF$-stable biset with large isotropy subgroups. We plan to study this in a future paper.

The paper is organized as follows: In Section \ref{section:ConstructionsOfFreeActions}, we recall the recursive method for constructing smooth actions that was introduced in \cite{unlu-yalcin2}. Theorem \ref{theorem:AbelianIsotropy} is proved in Section \ref{sect:Some new applications} and Theorem \ref{theorem:FreeActionsOfSolvableGroups} is proved in Section \ref{sect:solvable}.

\section{Constructions of smooth actions}
\label{section:ConstructionsOfFreeActions}

In this section, we summarize the recursive method that was introduced in \cite{unlu-yalcin2} for constructing free smooth actions on products of spheres. The main result of this section is Proposition \ref{prop:ConsSmoothActViaBisets} which is proved in Section 
\ref{subsect:Construction of group actions}. Most of the results in this section already appear in \cite{unlu-yalcin2}, but our presentation here is slightly different. We also introduce some new terminology which is used in the rest of the paper.

Let $G$ be a finite group and $M$ be a finite dimensional smooth manifold with a smooth $G$-action. Let $\Fa $ denote the family of isotropy subgroups of the $G$-action on $M$. 
Let $\G $ be a finite group  and
$${\mathbf A}=(\alpha _H)_{H\in \Fa }\in \lim _{\underset{H\in \Fa }{\longleftarrow}} \Rep (H,\Gamma )$$ be a \emph{compatible family of representations}. This means that
for every map $c_g: H\to K$ 
induced by conjugation with $g\in G$, 
there exists a $\gamma \in \G $ such that the following diagram commutes:
$$\xymatrix{
H \ar[d]^{c_g} \ar[r]^{\alpha _H}
& \Gamma \ar[d]^{ c_{\gamma }}  \\
K  \ar[r]^{\alpha _K}
&  \Gamma .}\\ $$
We sometimes write $(\alpha _H)$ instead of 
$(\alpha _H)_{H\in \Fa }$ to simplify the notation. We have the following geometric result which is the starting point for our constructions.

\begin{proposition}[Corollary 4.4 in \cite{unlu-yalcin2}]\label{prop:startingpoint}
Let $G$, $M$, $\Gamma$, and ${\mathbf A}=(\alpha_H)$ be as 
above. Given a unitary representation $\rho : \G \to U(n)$, 
there exist a positive integer $N$ and a smooth $G$-action 
on $M \times \bbS^N$ such that for every $x \in M$, the 
$G_x$-action on the sphere $\{x\}\times \bbS^N$ is 
diffeomorphic to the linear $G_x$-action on $\bbS(V^{\oplus 
k})$ where $V=\rho \circ \alpha_{G_x} $ and $k$ is some 
positive integer.
\end{proposition}

As discussed in \cite{unlu-yalcin2}, there is a
way to describe compatible representations using the terminology of fusion systems. 
Here we again consider fusion systems on any 
group (not necessarily a 
$p$-group, in fact, not necessarily a finite 
group).  

\begin{definition}\label{def:fusionsystems} 
Let $S$ be a group. A {\it fusion system} $\Fu$ on 
$S$ is a category whose objects are subgroups of $S$ and whose morphisms are injective group homomorphisms where the composition of morphisms in $\Fu $ is the usual composition of group homomorphisms and where for every $P,Q \leq S$, the morphism set $\Hom _{\Fu} (P,Q)$ satisfies the following:
\begin{enumerate}
\item $\Hom _S(P,Q) \subseteq \Hom _{\Fu} (P,Q)$ where $\Hom _S (P,Q)$ is the set of all conjugation homomorphisms induced by elements in $S$.
\item For every morphism $\varphi \in \Hom _{\Fu} (P,Q)$, the induced isomorphism $P\to \varphi(P)$ and its inverse are also morphisms in $\Fu$.
\end{enumerate}
\end{definition}

If $S$ is a subgroup of a group $\G $, then an 
obvious example of a fusion system is the fusion system $\Fu _S(\G )$ on $S$ for which the set of morphisms $\Hom _{\Fu} (P, Q)$ is defined as the set of all conjugations $c_{\gamma } (h)=\gamma h\gamma ^{-1}$ 
where  $\gamma \in \G $.

\subsection{Compatible homomorphisms} We now discuss another way to describe compatible representations using fusion systems. Let $G$ be a finite group and $\cH$ be a family of subgroups of $G$ closed under conjugation. Let $\Gamma$ be any group and $(\alpha_H)_{H \in \cH}$ be a family of homomorphisms $\alpha_H: H \to \Gamma$. Assume that $S$ is the subgroup of $\G $ generated by the union $$\displaystyle \cup \left\{\alpha_H(H)\,|\,H\in \cH \right\}.$$
Let $\{\iota_H : H \to S \ | \ H \in \cH \}$ be 
a family of maps such that $\alpha_H$ is equal to the composition $$ H \maprt{\iota _H } S \hookrightarrow \G.$$
Suppose that there is a fusion system $\cF$ on $S$ such that for every map $c_g: H \to K$ induced by conjugation between subgroups in $\Fa $, 
there is a monomorphism $f\in \cF$ such that the following diagram commutes
$$\xymatrix{ H \ar[d]^{c_g} \ar[r]^{\iota _H\ \ } & \iota _H (H) \ar@{>->}[d]^{f} \\
K  \ar[r]^{\iota _K\ \ } & \iota _K (K). }$$
Assume that $\cF$ is the smallest fusion system with these properties. Then one sees that 
the family $(\alpha_H)_{H \in \cH}$ is a compatible family of representations if and only if $\Fu \subseteq \Fu _S(\G )$.  

Before we continue our summary of the construction that we used in \cite{unlu-yalcin2}, we introduce a new terminology to 
give a common name to families of homomorphisms that satisfy properties 
like the families $(\alpha _H)$ and $(\iota _H)$ do satisfy.
We consider them as homomorphisms between families of 
subgroups in fusions systems. 

\begin{definition}\label{def:compatiblehomomorphism} Let $\cH $ be a family of objects in a fusion system 
$\Fu _G$ on a group $G$ and $\Fu _{ S }$ be a fusion system on a group $ S  $. For each $H\in \cH $, let $\iota  _H : H\to  S  $ be a group homomorphism. Then we say ${\mathbf 
A}=\left(\iota _H\right)_{H\in \cH }$ is a {\it compatible family of homomorphisms}  
from $\cF _G$ to  $\Fu _{ S }$ (supported by $\cH$) if for every $H,K\in \cH $, and every morphism $f: H \to K$ in $\Fu _G$ there exists a morphism 
$\widetilde{f}: \iota  _H(H)\to \iota  _K(K) $ in $\Fu _{ S }$ such 
that the following diagram
\begin{equation*}
\xymatrix{H \ar[d]^{f} \ar[r]^{\iota  _H \ \ \ }
& \iota  _H(H) \ar[d]^{\widetilde{f} }  \\
K  \ar[r]^{\iota  _K \ \ \ } & \iota  _K(K)}
\end{equation*}
commutes.  
\end{definition}
 
In particular, a compatible family of representations $(\alpha_H : H \to \G )_{H\in \cH}$ is a compatible family of homomorphisms from $\cF _G(G)$ to $\cF _{\G} (\G)$ supported by $\cH$.

\subsection{Construction of $\Gamma$}  This construction is the same as the construction given by Park \cite{park} for saturated fusion systems. Let $S$ be a finite group and $\Omega $ be an
$S$-$S$-biset. Let $\Gamma _{\Omega}$ denote the group of
automorphisms of the set $\Omega $ preserving the right $S$-action. In other words,
$$\Gamma _{\Omega }=\{f:\Omega \to \Omega \,|\, \forall s\in S, \,
\forall x \in \Omega, \, f(xs)=f(x)s\}.$$
Define $\iota :S\to \Gamma _{\Omega }$ as the homomorphism
satisfying $\iota (s)(x)=sx$ for all $x\in \Omega $. If the left
$S$-action on ${\Omega }$ is free and $\Omega $ is non-empty, then
$\iota $ is a monomorphism, hence we can consider $S$
as a subgroup of $\Gamma _{\Omega }$. 

We now introduce some terminology about bisets. 
Let $\Omega $ be an $S$-$S$-biset, $Q$ be a subgroup of $S$, and $\varphi : Q \to S$ be a monomorphism. Then, we write ${}_Q \Omega  $ to denote the
$Q$-$S$-biset obtained from $\Omega $ by restricting the left $S$-action to $Q$ and we write ${}_{\varphi} \Omega  $ to denote the
$Q$-$S$-biset obtained from $\Omega $ where the left $Q$-action is
induced by $\varphi $. A key lemma that we used in our previous paper 
\cite{unlu-yalcin2} is the following:

\begin{lemma}\label{lemma:FstableBisets}{\cite[Theorem 3]{park}} Let 
$\Omega $ be an $S$-$S$-biset with a free left $S$-action and let $Q$ be a
subgroup of $S$ and $\varphi :Q\to S $ be a monomorphism. Then,
 $ _{\varphi }{\Omega } $ and $ _{Q}{\Omega } $ are
isomorphic as $Q$-$S$-bisets if and only if $\varphi $ is a morphism
in the fusion system $\Fu _S (\Gamma _{\Omega }).$
\end{lemma}

We make the following definition for the situation considered in
Lemma \ref{lemma:FstableBisets}.

\begin{definition}\label{definition:FStableBiset}
Let $\Fu$ be a fusion system on a finite group $S$. Then, a left
free $S$-$S$-biset ${\Omega }$ is called {\it left} $\Fu $-{\it
stable} if for every subgroup $Q \leq S$ and $\varphi \in \Hom_{\Fu}
(Q, S)$, the $Q$-$S$-bisets ${}_Q {\Omega } $ and $ {}_{\varphi}
{\Omega }$ are isomorphic.
\end{definition}

Hence, by Lemma \ref{lemma:FstableBisets}, we have the following:

\begin{proposition}\label{prop:LeftFStable_FusionSystemInclusion}
Let $\Fu$ be a fusion system on a finite group $S$. If $\Omega $ is
a left $\Fu $-stable $S$-$S$-biset, then $\Fu\subseteq \Fu _S(\G
_{\Omega })$.
\end{proposition}

In \cite{park}, S. Park actually proves that $\cF$ is a full subcategory of $\cF (\G _{\Omega})$, i.e., $\cF=\cF _S (\Gamma _{\Omega} )$ when $\Gamma$ is a characteristic biset for the fusion system $\cF$. 

We now discuss a particular construction of a representation for $\Gamma _{\Omega}$ which comes from a representation of $S$.  

\subsection{Construction of a representation for $\Gamma$}
Let $V$ be a left $\bbC S$-module and $\Omega $ 
be an $S$-$S$-biset. We define the $\bbC \G_{\Omega}$-module $V _{\Omega}$ as follows
$$V _{\Omega}=\bbC \Omega \otimes _{\bbC S} V$$  
where $\bbC \Omega $ denotes the permutation $\bbC S$-$\bbC S$-bimodule
with basis given by $\Omega $. The left $\Gamma
_{\Omega}$-action on $\bbC \Omega $ is given by evaluation of the bijections in $\Gamma _{\Omega}$ at the elements of $\Omega $ and $V_{\Omega}$ is considered as a left $\bbC \Gamma _{\Omega}$-module via this action. 

Note that every transitive $S$-$S$-biset is of the form $S\times _{\Delta}S$ for some $\Delta \leq S \times S$, where $S\times _{\Delta} S $ is the set of equivalence classes of pairs $(s_1, s_2)$ where $(s_1t_1, s_2)\sim
(s_1, t_2s_2)$ if and only if $(t_1, t_2)\in \Delta$. The left and right
$S$-actions are given by the usual left and right multiplication in $S$. An $S$-$S$-biset is called bifree if both left and right $S$-actions are free. It is clear from the above description and from Goursat's lemma that a transitive $S$-$S$-biset $S \times _{\Delta} S$ is bifree if and only if  $\Delta$ is a graph of an injective map $\varphi: Q \to S$ where $Q \leq S$. In this case, we denote $\Delta$ by $$\Delta (\varphi )=\{ (s, \varphi(s) ) \ | \ s\in Q \}.$$ So, a bifree $S$-$S$-biset $\Omega$ is a disjoint union of bisets of the form $S\times _{\Delta (\varphi)} S $ where $\varphi : Q \to S$ is a monomorphism. We define $\Isot ({\Omega })$, the isotropy of ${\Omega }$, as the following set:
$$\Isot ({\Omega })=
\left\{ \varphi :Q\to S \left|\  S\times _{\Delta (\varphi)} S
\text{ is isomorphic to a transitive summand of }{\Omega } \right.
\right\}.$$

Every transitive biset can be written as a product of five basic bisets (see \cite[Lemma 2.3.26]{bouc}). Since
$\Omega$ is bifree, only three of these basic bisets, namely
restriction, isogation, and induction, are needed to express the
transitive summands of $\Omega$ as a composition of basic bisets.
By writing each transitive summand of $\Omega$ as a
composition of the three basic bisets, we can express
$\Res ^{\Gamma _{\Omega}}_{S} V_{\Omega}= \bbC  \Omega
\otimes_{\bbC S} V$ as a direct sum of $\bbC S$-modules of the form
$$\Ind ^{S}_{Q} \Isog ^* (\varphi ) \Res ^{S}_{\varphi (Q)} V$$ where
$\varphi:Q\to S$ is in $\Isot (\Omega)$. Note that $\Isog ^*
(\varphi)$ is the contravariant isogation defined by $\Isog ^*
(\varphi) (M)=\varphi^* (M)$ where  $M$ is a $\varphi (Q)$-module.

\subsection{Construction of group actions}\label{subsect:Construction of group actions} Let $V$ be a left $\bbC S$-module, $\Omega $ be a bifree
$S$-$S$-biset, and let $V _{\Omega}$ be the $\bbC \Gamma
_{\Omega}$-module constructed above. Then, for every $H\leq S$, the
$\bbC H$-module $\Res ^{\Gamma _{\Omega}}_{H} V _{\Omega}$ is a direct sum of modules in the form
$$\Ind ^{H}_{H\cap  Q^x} \Isog ^*(\varphi \circ c_{x}) \Res
^{S}_{\varphi (\leftexp{x}H\cap Q)} V$$ where $x\in S$ and
$\varphi:Q\to S$ is in $\Isot (\Omega )$ (see \cite[Proposition 5.8]{unlu-yalcin2}). Now, we state the main result of this section.

\begin{proposition} \label{prop:ConsSmoothActViaBisets}
Let $G$ be a finite group, $M$ be a finite dimensional smooth manifold with a smooth $G$-action, and $\Fa$ denote the family of isotropy subgroups of the $G$-action on $M$. Let $\cF $ be a fusion system on a finite group $S$ and $\Omega$ be a left $\cF $-stable $S$-$S$-biset. 

Then, given a compatible family of homomorphisms $\left(\iota _H\right)_{H\in \Fa}$ from $\cF _G(G)$ to $\cF $ supported by $\Fa$  and a $\mathbb{C}S$-module $V$, we can construct a smooth $G$-action on $M \times \bbS^N$ for some positive integer $N$ so that a subgroup $K \leq G$ fixes a point on $M \times \bbS^N$ if and only if there exists $H \in \Fa$, $x \in S$, and $\varphi : Q \to S$ in $\Isot (\Omega)$ such that $K \leq H$ and $\Res ^S _{\varphi(\leftexp{x}L \cap Q)} V$ has a trivial summand where $L=\iota _H (K)$.
\end{proposition}

\begin{proof} By Proposition \ref{prop:startingpoint}, there exists a $G$-action on $M \times \bbS^N$ for some $N$ so that for every $x\in M$ the $G_x$-action on $\bbS^N$ is diffeomorphic to the linear action coming from the representation $V_{\Omega}$ of $\Gamma _{\Omega}$. Assume $K \leq G$ fixes a point on $M \times \bbS^N$. Then $K\leq H$ for some $H \in \cH $. Let $L=\iota _H (K)$. Then, the  $\bbC L$-module $\Res ^{\Gamma _{\Omega}}_{L} V _{\Omega}$ is a
direct sum of modules in the form
$$\Ind ^{L}_{L\cap  Q^x} \Isog ^*(\varphi \circ c_{x}) \Res
^{S}_{\varphi (\leftexp{x}L\cap Q)} V$$ where $x\in S$ and
$\varphi:Q\to S$ is in $\Isot (\Omega )$. Hence as a $\bbC K$-module, 
$\left(\iota _H\right)^* \left(\Res ^{\Gamma _{\Omega}}_{L} V _{\Omega}\right)$ is a
direct sum of modules of the form
\begin{equation}\label{eqn:composition}
\Inf ^{K}_{K/J}\Isog ^*(\iota _H )\Ind ^{L}_{L\cap  Q^x} \Isog ^*(\varphi \circ c_{x}) \Res
^{S}_{\varphi (\leftexp{x}L\cap Q)} V
\end{equation} 
where $J=\ker(\iota _H)\cap K$, $x\in S$, and
$\varphi:Q\to S$ is in $\Isot (\Omega )$. A $\bbC K$-module of the form $(\ref{eqn:composition})$ has a trivial summand if and only if $\Res ^S _{\varphi (\leftexp{x} L \cap Q)} V$ has a trivial summand.
Therefore $K\leq G$ fixes a point on $M \times \bbS^N$ if and only if  there exists $H \in \Fa$, $x\in S$, and
$\varphi:Q\to S$ in $\Isot (\Omega )$ 
such that $K \leq H$ and $\Res
^{S}_{\varphi (\leftexp{x}L\cap Q)} V$
has a trivial summand. 
\end{proof}

\section{Constructing $\cF$-stable bisets}
\label{sect:Some new applications}

In the previous section we summarized the method for constructing smooth actions using fusion systems. The main  result of this method is stated as Proposition \ref{prop:ConsSmoothActViaBisets}. This proposition 
suggests that to construct new smooth actions on products of spheres, we need to understand how to construct left $\cF$-stable bisets with large isotropy subgroups. Note that for a fusion system $\cF$ on a finite group $S$, the $S$-$S$-biset $S\times S$ is a left $\cF$-stable biset. However it is clear that this biset will not be very useful for constructing free actions.  

In this section, we prove a proposition that is very useful for constructing left $\cF$-stable bisets with large isotropy. We start with a definition.

\begin{definition}\label{def:stronglyclosed} Let $\Fu $ be a fusion system on a finite group $S$.
Then we say $K$ is a {\it strongly} $\Fu $-{\it closed} subgroup of $S$ if
for any subgroup $L\leq K$ and for any morphism $\varphi : L \to S $
in $\Fu$ we have $\varphi (L)\leq K$
\end{definition}

Note that if $K$ is a strongly $\cF$-closed subgroup of $S$, then the fusion system $\cF$ restricts to a fusion system on $K$. We denote this fusion system by $\cF |_{K}$. Now we are ready to state our main result in this section.

\begin{proposition}\label{prop:Fchar_leftFbiset} Let $\Fu $ be a fusion system on a finite group $S$ and $K$ be a strongly $\Fu $-closed subgroup
of $S$. Let $\Omega '$ be a left $\Fu | _{K}$-stable 
$K$-$K$-biset. Then, the $S$-$S$-biset 
$$\Omega = S\times _{K} \Omega ' \times _{K} S$$
is left $\Fu $-stable. Moreover, if $\varphi: Q \to S$ is in $\Isot (\Omega)$, then $\varphi$ can be expressed as a composition $$Q \maprt{\varphi'} K\hookrightarrow S$$ for some  $\varphi ' : Q \to K$ in $\Isot (\Omega ')$.
\end{proposition}

\begin{proof}  
 Note that the biset $\Omega$ is defined as a composition of three bisets, so we can write $\Omega=\Ind _K ^S \Omega ' \Res ^S _K $ where $\Ind ^S _K $ denotes the $S$-$K$-biset ${}_S S_K$ and $\Res ^S _K$ denotes the $K$-$S$-biset ${}_K S_S$. Given a morphism $\varphi :Q\to S$ in $\cF $, we have 
\begin{equation*} {}_{\varphi } \Omega = \Isog ^* (\varphi ) \Res ^{S}_{\varphi (Q)} \Omega  = \Isog ^* (\varphi ) \Res ^{S}_{\varphi (Q)} \Ind ^{S}_{K} \Omega ' \Res ^{S}_{K}. \end{equation*} 
Applying the Mackey formula to $\Res ^{S}_{\varphi (Q)} \Ind ^{S}_{K}$, we get  
\begin{eqnarray*} {}_{\varphi}\Omega &=& \underset{\varphi (Q)xK\in \varphi (Q)\backslash  S/K}{\coprod}   \Isog ^* (\varphi ) \Ind ^{\varphi (Q)}_{\varphi (Q)\cap K^x}\Isog^* (c_{x})\Res ^{K}_{\leftexp{x}\varphi (Q)\cap K} \Omega ' \Res ^{S}_{K} \\ &=&   \underset{\varphi (Q)xK\in \varphi (Q)\backslash  S/K}{\coprod}   \Ind ^{Q}_{Q\cap \varphi ^{-1}(K^x)}\Isog ^* (\varphi ) \Isog^* (c_{x})\Res ^{K}_{\leftexp{x}\varphi (Q)\cap K} \Omega ' \Res ^{S}_{K}  \end{eqnarray*}
where the second equality comes from the commutativity of isogation and induction. Now, since $\Omega '$ is a left $\Fu | _{K}$-stable 
$K$-$K$-biset and $K$ is a strongly $\cF$-closed subgroup, we obtain
\begin{eqnarray*}  {}_{\varphi } \Omega  & = &   \underset{\varphi (Q)xK\in \varphi (Q)\backslash  S/K}{\coprod}   \Ind ^{Q}_{Q\cap \varphi ^{-1}(K^x)}\Res ^{K}_{Q\cap \varphi ^{-1}(K^x)} \Omega ' \Res ^{S}_{K} \\  &= &   \underset{\varphi (Q)xK\in \varphi (Q)\backslash  S/K}{\coprod}   \Ind ^{Q}_{Q\cap K}\Res ^{K}_{Q\cap K} \Omega ' \Res ^{S}_{K}. 
\end{eqnarray*}

Notice that in the disjoint union above we are taking a disjoint union of $S$-$S$-bisets which do not depend on the double coset $\varphi (Q)xK$. So we can write  $${}_{\varphi} \Omega=n_{\varphi} \Ind ^{Q}_{Q\cap K}\Res ^{K}_{Q\cap K} \Omega ' \Res ^{S}_{K}$$ where $n_{\varphi} =|\varphi(Q)\backslash S / K|$. Note that this computation holds for any $\varphi : Q \to S$, in particular, for the inclusion map ${\rm inc} : Q\hookrightarrow S$. So, to prove that ${}_{\varphi} \Omega \cong {}_Q \Omega$ as $Q$-$S$-bisets it is enough to prove the equality $|\varphi (Q)\backslash S/K|=|Q\backslash S /K|$. 

Since $K$ is strongly $\cF $-closed, $K$ is normal in $S$. So, the number of double cosets $|\varphi (Q)\backslash  S/K|$ is equal to $|S/\varphi (Q)K|$. 
Therefore to prove the above equality, it is enough to show $|\varphi (Q)\cap K|=|Q\cap K|$. Note that $\varphi |_{Q \cap K} : Q \cap K \to S$ is a morphism in $\cF$ and $K$ is strongly $\cF$-closed, so the image of $\varphi |_{Q \cap K}$ must lie in $K$. This gives $\varphi (Q \cap K)\leq \varphi (Q)\cap K$, thus $|Q \cap K|\leq |\varphi(Q)\cap K|$. For the inverse inequality, apply the same argument to $\varphi^{-1} : \varphi (Q) \to S$. This completes the proof of the first part of the proposition.

For the second part, observe that $\Omega '$ can be expressed as a disjoint union of bisets of the form 
$$\Ind ^{K}_{Q} \Isog ^* (\varphi ' ) \Res ^{K}_{\varphi ' (Q)}$$ where the union is over morphisms
$\varphi ' :Q\to K$ in $\Isot (\Omega ')$ (see \cite[Lemma 2.3.26]{bouc}). Since
$\Omega= \ind _K ^S \Omega ' \Res ^S _K$, we obtain that
$$ \Omega = \ind _K ^S \Ind ^{K}_{Q} \Isog ^* (\varphi ' ) \Res ^{K}_{\varphi ' (Q)} \Res ^S _K=\Ind ^S _Q \Isog ^* (\varphi ') \Res ^S _{\varphi ' (Q))}$$ using the composition properties of restriction and induction bisets. So, we can conclude that every $\varphi \in \Isot(\Omega)$ can be expressed as a composition $\varphi: Q \maprt{\varphi'} K\hookrightarrow S$ for some  $\varphi ' : Q \to K$ in $\Isot (\Omega ')$. 
\end{proof}

Now Theorem \ref{theorem:AbelianIsotropy} follows from a recursive application of the following proposition.

\begin{proposition}\label{prop:AbelianIsotropy}
Let $k\geq 1$ be an integer and $G$ be a finite group acting smoothly on a manifold $M$. If all
the isotropy subgroups of $M$ are abelian groups with
rank $\leq k$, then $G$ acts freely and smoothly on $M \times \bbS^{n}$ for some positive
integer $n$ such that all the isotropy subgroups of $M\times \bbS^{n}$ are abelian groups with rank $\leq k-1$ .
\end{proposition}

\begin{proof} Let $G$ be a finite group acting smoothly on a manifold $M$. Let  $\Fa$ denote the family of
isotropy subgroups of the $G$-action on $M$. Assume that all the subgroups $H \in \Fa$ are abelian with rank $\leq k$. Then there exists a finite abelian 
group $S$ of rank $k$ such that for every $H \in \Fa$ there is a monomorphism $\iota _H : H \to S$. Take $\Fu $ to be the largest possible fusion system on $S$. Let $K$ be the subgroup of $S$ defined as follows
$$K=\left\{x\in S\, | \, \text{order of } x \text{ is a square free product of primes} \right\}.$$
It is clear that $K$ is a strongly $\cF$-closed subgroup of $S$. This allows us to restrict the fusion system $\cF$ to $K$. Let $\cF_K$ denote the restriction of $\cF$ to $K$.

We claim that $K$ is also $\cF$-characteristic \cite[Definition 6.2]{unlu-yalcin2}, i.e.,  for any $L\leq K$ and for any morphism $\varphi : L \to K $
in $\Fu _K$, there exists a morphism $\widetilde{\varphi }:K\to K$  in $\Fu _K$ such that $\widetilde{\varphi }(l)=\varphi (l)$ for all $l \in
L$. To see this, note that $K$ is a direct product of elementary abelian $p$-subgroups, so for any $L \leq K$, there are direct sum decompositions $K=L\oplus J=\varphi(L)\oplus J'$ where $J, J' \leq K$. Since $J\iso J'$, we can choose an isomorphism $\psi : J\to J'$ and define $\widetilde \varphi :K \to K$ as the map $\widetilde \varphi (l,j)=(\varphi (l), \psi (j))$ for every $l\in L$ and $j\in J$. So, $K$ is $\cF$-characteristic.  

Now, since $K$ is $\cF$-characteristic,  the $K$-$K$-biset 
$$\Omega '=\coprod _{\varphi \in \Aut _{\Fu _K}(K)} K\times _{\Delta(\varphi )} K$$
is left $\Fu _K$-stable (see \cite[Lemma 6.3]{unlu-yalcin2}). Using Proposition \ref{prop:Fchar_leftFbiset} we conclude that the $S$-$S$-biset
$$\Omega = S\times _{K} \Omega ' \times _{K} S$$
is left $\Fu $-stable. 

Let $V$ be a one-dimensional $\mathbb{C}S$-module which is not trivial when it is 
restricted to any Sylow $p$-subgroup of $K$. By Proposition 
\ref{prop:ConsSmoothActViaBisets}, $G$ acts smoothly on $M \times \bbS^N$ for some $N$ so that if $U\leq G$ fixes a point $M \times \bbS^N$, then $U\leq H$ for some $H \in \Fa$ and  $$ W=\Res ^{S}_{\varphi (\iota _H (U)\cap Q)} V$$ is trivial for some $\varphi : Q \to S$ in $\Isot (\Omega)$.
By Proposition \ref{prop:Fchar_leftFbiset}, every morphism $\varphi : Q \to S$ in $\Isot (\Omega)$ is of the form $\varphi : K \to K \hookrightarrow S$, in particular, $Q=K$. If $U$ has rank equal to $k$, then $\varphi (\iota _H (U)\cap K)$ includes a Sylow $p$-subgroup of $K$, so $W$ is not trivial for such subgroups. This completes the proof. 
\end{proof}

\section{Free actions of solvable groups}
\label{sect:solvable}

We now consider the problem stated in the introduction
which asks whether or not every finite group can act freely on some product of spheres with trivial action on homology.
We first introduce some terminology which we believe is useful for studying this problem  and then we prove Theorem \ref{theorem:FreeActionsOfSolvableGroups} which solves the problem for finite solvable groups. 

Let $\Gamma _i$ be a group for $i=1,2$ and $\Fu _{\Gamma _i}$ be a fusion system on the group $\Gamma _i$ for $i=1,2$. Let  ${\mathbf A}=\left(\alpha_H\right)_{H\in \cH }$ be a compatible family of homomorphisms from $\cF _G (G)$ to $\Fu _{\Gamma_1}$ supported by $\Fa$ and ${\mathbf B}=\left(\beta _K\right)_{K\in \cK }$ be a compatible family of homomorphisms from 
$\Fu _{\Gamma_1}$ to $\Fu _{\Gamma_2}$ supported by $\cK$. If for every $H \in \cH$, we have $\alpha _H (H)\in \cK$, then we can compose ${\mathbf 
A}$ with ${\mathbf B}$ and obtain $${\mathbf B}\circ {\mathbf A}=\left(\beta _{\alpha_H(H)}\circ \alpha_H\right)_{H\in \cH }$$ which is a compatible family of homomorphisms from $\cF _G (G) $ to $\Fu 
_{\Gamma_2}$ supported by $\cH$.

Note that if $\rho: \Gamma _1\to \Gamma _2$ is a group homomorphism, then $\rho $ induces a compatible family of homomorphism from $\Fu _{\Gamma _1}(\Gamma _1)$ to $\Fu _{\Gamma _2}(\Gamma _2)$. So, now we can give the following definition.

\begin{definition} Let ${\mathbf V}=(V_H)_{H\in \Fa}$ be a compatible family of unitary representations. We say ${\mathbf V}$ factors 
through a group $\Gamma $ if there exists a homomorphism $\rho : \Gamma 
\to U(n)$ and a compatible family of homomorphisms ${\mathbf A}$ from $\cF _G (G)$ to $\cF _{\G} (\G)$ such that ${\mathbf V}=\rho \circ {\mathbf A}$
\end{definition}

Now we can rephrase Proposition \ref{prop:startingpoint} in the following way. 

\begin{proposition}\label{proposition:ConstructingSmoothActions}
Let $G$ be a finite group and $M$ be a finite dimensional smooth manifold with a smooth $G$-action. Let $\cH $ denote the family of isotropy subgroups of the $G$-action on $M$. Let 
$ {\mathbf V}=(V _H)_{H \in \Fa}$ be a compatible family of unitary representations. If $\, 
{\mathbf V}$ factors through a finite group, then there exists a smooth $G$-action on $M \times \bbS^N$ for some $N$ such that
for every $x \in M$, the $G_x$-action on the sphere $\{x\}\times
\bbS^N$ is given by the linear $G_x$-action on $\bbS(V_{G_x} ^{\oplus k})$ for some $k$. Moreover if the $G$-action on $M$ is homologically trivial, then there exists an $N$ such that the $G$-action on $M \times \bbS ^N$ is also homologically trivial.
\end{proposition}

\begin{proof} We only need to prove the last statement. The first part is already proved in Proposition \ref{prop:startingpoint}. Note that the $G$-space $M \times \bbS ^N$ is constructed as the total space of a sphere bundle of a  $G$-equivariant  orientable vector bundle $p: E \to M$. So, we need to show that there is an $N>0$ such that the $G$-action on the homology of $SE$ is trivial where $SE$ denotes the total space of the sphere bundle of $p$.  We will show this  using the Gysin sequence which relates the homology of $M$ to the homology of $SE$.  Recall that the Gysin sequence for a sphere bundle is obtained from the long exact sequence of the pair $(DE, SE)$ where $DE$ denotes the total space of the disk  bundle of $p$. In particular, we have the following diagram
\begin{equation*}
\xymatrix{\dots \ar[r]  & H_n (SE) \ar[r] & H_n (DE) \ar[d]_{\cong} ^{p_*} \ar[r] & H_n (DE, SE) \ar[d]_{\cong}^{\Phi } \ar[r]^-{\partial} &H_{n-1} (DE) \ar[r] & \dots \\ & &H_n (M)  & H_{n-N-1}(M) & &}
\end{equation*}
where the homomorphism $\Phi$ is defined by $\Phi (z)=p_* (U_p \cap z)$ for every $z\in H_n(DE, SE)$. Here $U_p \in H_{N+1} (D(E), S(E))$ denotes the Thom class of $p$ (see \cite[pg. 260]{spanier}). Note that if the Thom class is invariant under the $G$-action, then we can conclude that the Gysin sequence is a sequence of $\bbZ G$-modules since the rest of the above diagram is formed by $\bbZ G$-modules.

The $G$-action on the Thom class may be nontrivial in general, but for every $g\in G$, we always have $gU_p=\pm U_p$. The Thom class of the Whitney sum of two vector bundles is the cup product of Thom classes of corresponding bundles (see \cite[Thm. 2]{holm}). So, by taking the Whitney sum of $p$ with itself, we can assume that the $G$-action on the Thom class is trivial. This means that we can assume that the Thom isomorphism $\Phi : H_n (DE, SE) \maprt{\cong} H_{n-N-1} (M)$ is an isomorphism of $\bbZ G$-modules and hence the Gysin sequence is a sequence of $\bbZ G$-modules. 

By taking further Whitney sums if necessary, we can also assume that $N >\dim M$. Then, we will have $H_n (SE) \cong H_n (M)$ for all $n<N$ and $H_n (SE)\cong H_{n-N} (M)$ for all $n \geq N$. These isomorphisms are isomorphisms of $\bbZ G$-modules and since it is assumed that the $G$-action on the homology of $M$ is trivial,  we can conclude that the $G$-action on the homology of $SE$ is trivial.  
\end{proof}

Note that in the above proposition, if $H$ is  a maximal group  in $\cH 
$ and if $V _H$ has no trivial summands, then no point on $M \times \bbS^N$ will be fixed by $H$. Hence the total size of isotropy subgroups of the $G$-action on $M \times \bbS^N$ would be smaller.

\begin{definition}\label{defn:reducible} Let $\cH _1$ and $\cH _2$ be two families of 
subgroups in a group $G$ which are closed under conjugation and taking subgroups.  Let $\Fu $ be a fusion system on the group 
$G$. Then we say $\cH _2$ is reducible to $\cH _1$ if there 
exists a compatible family of unitary representations ${\mathbf V}=(V _H)_{H \in \Fa}$ of $\cH _2$ such that $\mathbf V$ factors through a finite group and $\cH _1$ is equal to the family of 
all subgroups of isotropy groups of the $G$-action on the disjoint union of the left 
$G$-sets $G\times _H \bbS(V _H)$ over all $H \in \cH _2$. 
\end{definition}

Let $\cH _1$ and $\cH _2$ be two families of subgroups of $G$ which are closed under conjugation and taking subgroups. 
Then we write $\cH _1 \leq \cH _2 $ if every subgroup in $\cH _1$ 
is contained in some subgroup in $\cH _2$. Note that if $\cH _2$ is 
reducible to $\cH _1$ then $\cH _1 \leq \cH _2 $. Moreover we write $\cH 
_1<\cH _2 $ if $\cH _1 \leq \cH _2 $ and $\cH _1\neq \cH _2$. In 
particular, this means that there exists a maximal subgroup in $\cH _1$ which is not maximal in $\cH _2 $.

\begin{definition} Let $\Fu $ be a fusion system on a group $G$. A 
sequence $$\cK=\cH _0 < \cH _1 < \dots <\cH _n=\cH$$ of families of subgroups in 
$G$ closed under conjugation and taking subgroup is called a {\it reduction sequence} of length $n$ from $\cH$ to $\cK$ if the family $\cH _{i}$ is reducible to $\cH _{i-1}$ for all $i\in \{1,2,\dots ,n\}$.
Moreover in this case we say $\cH$  is reducible to $\cK$ in $n$ steps.
\end{definition}

If there is a reduction sequence in $\cF$ from $\cH$ to $\cK=\{1\}$, then we say the family $\cH$ is reducible. If $\cH$ is  the family of all subgroups of $G$, then we say the fusion system $\cF$ is reducible. If $\Fu =\cF_G (G) $ is a reducible fusion system, we say the group $G$ is reducible.  The smallest number of steps required to reduce a family $\Fa$, a fusion system $\cF$, or a group $G$ are denoted by $n_{\Fa}$, $n_{\cF}$, and $n_G$, respectively. 

\begin{proposition}\label{prop:FreeActionsOn_nG_spheres}
Let $G$ be a finite group. If $G$ is a reducible group then $G$ can act freely, smoothly, and homologically trivially on a product of $n_G$ spheres.
\end{proposition}

\begin{proof} Let $G$ be reducible. Take a reduction 
sequence $$\{ 1\}=\cH _0 < \cH _1 < \dots <\cH _n=\cH$$ of families of subgroups in $G$ such that $n=n_G$ and  $\cH$ is the family of all subgroups in $G$. Note that $G$ acts on $M=pt$ smoothly and homologically trivially with isotropy subgroups in $\cH$. Now assume that $G$ acts smoothly on some smooth manifold $M$ homologically trivially with isotropy subgroups in $\cH _i$ for some $0\leq i \leq n$. By Definition  \ref{defn:reducible}, there exists a compatible family of unitary representations $(V_H)_{H\in \cH _i}$ such that subgroups that   fixes a point in $G \times _H \bbS (V_H)$ are in $\cH _{i-1}$ for every $H \in \cH_i$. Now by Proposition \ref{proposition:ConstructingSmoothActions}, we obtain a smooth $G$-action on $M \times \bbS ^N$ for some $N$ with trivial action on homology such that all the isotropy subgroups are in $\cH _{i-1}$. Continuing this way, we get a smooth free action on a product of $n$ spheres with trivial action on homology.
\end{proof}

To prove Theorem \ref{theorem:FreeActionsOfSolvableGroups} we will show that all solvable groups are reducible. The proof follows from a construction that is similar to the construction used in the previous section.  

\begin{proposition} Let $G$ be a finite solvable group. Then, $G$ is reducible.
\end{proposition}

\begin{proof} Let $G$ be a finite solvable group and $G^{(0)}=G$ and $G^{(i+1)}=[G^{(i)},G^{(i)}]$ for $i\geq 0$.
Let $\Fa $ be a family of subgroups in $G$ closed under conjugation and taking subgroups. Assume that $\cH \neq \{1\}$ and  $n$ is the largest integer such that $H \leq G^{(n)}$ for every $H \in \cH$. Let $$r=\max \, \{ \rk (HG^{(n+1)}/G^{(n+1)}) \ | \ H \in \cH\}$$
and let $S= (\bbZ /m)^r$ where $m$ is equal to the exponent of $G^{(n)}/G^{(n+1)}$. For each $H \in \cH$, choose a map $j_H : H G^{(n+1)}/G^{(n+1)} \to S$ and define $\iota _H:H \to S$ as the composition 
$$H\to H/H\cap G^{(n+1)}\cong HG^{(n+1)}/G^{(n+1)}\maprt{j_H}S$$ for every $H \in \Fa $. Let $\cF _S$ denote the largest possible fusion system on $S$. Then,  $(\iota _H)_{H\in \cH}$ is a compatible family of homomorphisms from $\cF _G (G)$ to $\cF _S$ supported by $\cH$. Using the biset and the representation constructed in Proposition \ref{prop:AbelianIsotropy}, we can reduce $\Fa $ to $\Fa '$ so that every $H \in \Fa '$ satisfies $\rk (HG^{(n+1)}/G^{(n+1)})\leq r-1$. In particular, $\cH '$ is strictly smaller than $\cH$. This shows that $G$ is reducible. \end{proof}

This completes the proof of Theorem \ref{theorem:FreeActionsOfSolvableGroups}. It is clear from the above proof that if we can find ways to construct $\cF$-stable bisets with large isotropy subgroups then we can prove more general theorems on constructions of free actions and possibly solve Problem \ref{prob:MainProblem} completely. In a future work we will investigate which conditions on a fusion system $\cF$ guarantee the existence of a left $\cF$-stable biset with relatively large isotropy subgroups.

\acknowledgement{We thank the referee for a careful reading of the paper and for many corrections and helpful comments.}

\providecommand{\bysame}{\leavevmode\hbox
to3em{\hrulefill}\thinspace}
\providecommand{\MR}{\relax\ifhmode\unskip\space\fi MR }
\providecommand{\MRhref}[2]{%
  \href{http://www.ams.org/mathscinet-getitem?mr=#1}{#2}
} \providecommand{\href}[2]{#2}

\end{document}